\documentclass[12pt]{article}
\usepackage[misc]{ifsym}
\usepackage[T1]{fontenc}
\usepackage[utf8]{inputenc}
\usepackage{authblk}
\usepackage{indentfirst}                                                                             \usepackage[top=1in, bottom=1in, left=1.00in, right=1.00in]{geometry}

\usepackage{amsmath}
\usepackage{enumerate}
\usepackage{amsthm}
\usepackage{abstract}

\usepackage{bm}
\usepackage{graphics}
\usepackage{graphicx}
\usepackage{subfigure}
\usepackage{caption}
\usepackage{amsthm,amsmath,amssymb}
\usepackage{float}
\usepackage[section]{placeins}
\usepackage{multirow}
\usepackage{amssymb}
\usepackage{IEEEtrantools}
\usepackage{cite}
\usepackage{pifont}
\pdfoutput=1

\title{Tur$\rm{\acute{a}}$n problem for $\mathcal{K}_4^-$-free signed graphs}
\date{}
\author{Fan Chen$^{1,\, 2}$, Xiying Yuan$^{1,\, 2}$\thanks{Corresponding author. Email address: xiyingyuan@shu.edu.cn (Xiying Yuan)\\ \indent chenfan@shu.edu.cn (Fan Chen)\\ \indent This work is supported by the National Nature Science Foundation of China (Nos.11871040,
12271337)}}
\affil{{\footnotesize\emph{1. Department of Mathematics, Shanghai University, Shanghai 200444, P.R. China}}\\
{\footnotesize\emph{2. Newtouch Center for Mathematics of Shanghai University, Shanghai 200444, P.R. China}}}

\begin{document}
\newtheorem{theorem}{Theorem}[section]
\newtheorem{assumption}[theorem]{Assumptio}
\newtheorem{corollary}[theorem]{Corollary}
\newtheorem{proposition}[theorem]{Proposition}
\newtheorem{lemma}[theorem]{Lemma}
\newtheorem{definition}[theorem]{Definition}
\newtheorem{remark}[theorem]{Remark}
\newtheorem{problem}[theorem]{Problem}
\newtheorem{claim}{Claim}
\newtheorem{conjecture}[theorem]{Conjecture}
\newtheorem{fact}{Fact}

\maketitle
\noindent\rule[0pt]{16.5cm}{0.09em}

\noindent{\bf Abstract}

\noindent Suppose that $\dot{G}$ is an unbalanced signed graph of order $n$ with $e(\dot{G})$ edges. Let $\rho(\dot{G})$ be the spectral radius of $\dot{G}$, and $\mathcal{K}_4^-$ be the set of the unbalanced $K_4$. In this paper, we prove that if $\dot{G}$ is a $\mathcal{K}_4^-$-free unbalanced signed graph of order $n$, then $e(\dot{G})\leqslant \frac{n(n-1)}{2}-(n-3)$ and $\rho(\dot{G})\leqslant n-2$. Moreover, the extremal graphs are completely characterized.


\noindent\rule[0pt]{16.5cm}{0.05em}

\section{Introduction}
The graph $G$ is considered to be simple and undirected throughout this paper. The vertex set and the edge set of a graph $G$ will be denoted by $V(G)$ and $E(G)$. A signed graph $\dot{G}=(G,\sigma)$ consists a graph $G$, called the underlying graph, and a sign function $\sigma:E(G) \rightarrow \left\{-1,+1\right\}$. The signed graphs firstly appeared in the work of Harary \cite{H}. If all edges get signs $+1$ (resp. $-1$), then $\dot{G}$ is called all positive (resp. all negative) and denoted by $(G,+)$ (resp. $(G,-)$). The sign of a cycle $C$ of $\dot{G}$ is $\sigma(C)=\prod_{e\in E(C)}\sigma(e)$, whose sign is $+1$ (resp. $-1$) is called positive (resp. negative). A signed graph $\dot{G}$ is called balanced if all its cycles are positive; otherwise it is called unbalanced. For more details about the notion of signed graphs, we refer to \cite{Z}.

Let $U$ be a subset of the vertex set $V(\dot{G})$ and $\dot{G}_U$ be the signed graph obtained from $\dot{G}$ by reversing the sign of each edge between a vertex in $U$ and a vertex in $V(\dot{G})\setminus U$. We say the signed graph $\dot{G}_U$ is switching equivalent to $\dot{G}$, and write $\dot{G}\sim \dot{G}_U$. The switching operation remains the signs of cycles. So if $\dot{G}$ is unbalanced, and $\dot{G}_U$ is also unbalanced.

For an $n\times n$ real symmetric matrix $M$, all its eigenvalues will be denoted by $\lambda_1(M)\geqslant\lambda_2(M)\geqslant\cdots\geqslant\lambda_n(M)$, and we write Spec($M$)=$\{\lambda_1(M),\lambda_2(M),\cdots,\lambda_n(M)\}$ for the spectra of $M$. The adjacency matrix of a signed graph $\dot{G}$ of order $n$ is an $n\times n$ matrix $A(\dot{G})=(a_{ij})$. If $\sigma(uv)=+1$ (resp. $\sigma(uv)=-1$), then $a_{uv}=1$ (resp. $a_{uv}=-1$) and if $u$ is not adjacent to $v$, then $a_{uv}=0$. The eigenvalues of $A(\dot{G})$ are called the eigenvalues of $\dot{G}$, denoted by $\lambda_1(\dot{G})\geqslant\lambda_2(\dot{G})\geqslant\cdots\geqslant\lambda_n(\dot{G})$. In particular, the largest eigenvalue $\lambda_1(\dot{G})$
is called the index of $\dot{G}$. The spectral radius of $\dot{G}$ is defined by $\rho(\dot{G})=\max\big\{|\lambda_i(\dot{G})|:1\leqslant i\leqslant n\big\}.$
Since, in general, $A(\dot{G})$ is not similar to a non-negative matrix, it may happen that $-\lambda_n(\dot{G})>\lambda_1(\dot{G})$. Thus,
$\rho(\dot{G})=\max\big\{\lambda_1(\dot{G}),-\lambda_n(\dot{G})\big\}.$ For the diagonal matrix $S_U=\text{diag}(s_1,s_2,\cdots,s_n)$, we have $A(\dot{G})=S_U^{-1}A(\dot{G}_U)S_U$ where $s_i=1$ if $i\in U$, and $s_i=-1$ otherwise. Therefore, the signed graphs $\dot{G}$ and $\dot{G}_U$ share the same spectra.

A graph may be regarded as a signed graph with all positive edges. Hence, the properties of graphs can be considered in terms of signed graphs naturally. Moreover, there are some special properties in terms of signed graphs. Such as Huang \cite{H1} solved the Sensitivity Conjecture by the spectral properties of signed hypercubes. For the spectral theory of signed graph, see \cite{BCKW ,KP, KS, ABH} for details, where \cite{BCKW} is an excellent survey about some open problems in the spectral theory of signed graphs.

Let $\mathcal{F}$ be a family of graphs. Graph $G$ is $\mathcal{F}$-free if $G$ does not contain any graph in $\mathcal{F}$ as a subgraph. The classical Tur$\rm{\acute{a}}$n type problem determines the maximum number of edges of an $n$ vertex $\mathcal{F}$-free graph, called the Tur$\rm{\acute{a}}$n number.
Let $T_r(n)$ be a complete $k$-partite graph of order $n$ whose partition sets
have sizes as equal as possible. Tur$\rm{\acute{a}}$n \cite{T} proved that $T_r(n)$ is the unique extremal graph of $K_{r+1}$-free graph, which is regarded as the beginning of the extremal graph theory. We refer the reader to \cite{BS, FG, YZ} for more results about Tur$\rm{\acute{a}}$n number.
\begin{theorem}\cite{T}
If $G$ is a $K_{r+1}$-free graph of order $n$, then
$$e(G)\leqslant e(T_r(n)),$$
with equality holding if and only if $G=T_r(n)$.
\end{theorem}

In 2007, Nikiforov \cite{N1} gave a spectral version of the Turán Theorem for the complete graph $K_{r+1}$. In
the past few decades, much attention has been paid to the search for the spectral Turán Theorem such as \cite{CDT, DKL,WKX}.
\begin{theorem}\cite{N1}
If $G$ is a $K_{r+1}$-free graph of order $n$, then
$$\rho(G)\leqslant\rho(T_r(n)),$$
with equality holding if and only if $G=T_r(n)$.
\end{theorem}

How about the Tru$\rm{\acute{a}}$n problem of signed graph? The discussions about the complete signed graph kicked off. Let $\mathcal{K}_3^-$ be the set of the unbalanced $K_3$. Up to switching equivalence, we have $\mathcal{K}_3^-=\{\dot{H}\}$, where $\dot{H}$ is the signed triangle with exactly one negative edge. Wang, Hou, and Li \cite{WHL} determined the Tur$\rm{\acute{a}}$n number of $\mathcal{K}_3^-$ and the spectral Tur$\rm{\acute{a}}$n number of $\mathcal{K}_3^-$. The dashed lines indicate negative edges, and ellipses indicate the cliques with all positive edges in Fig. 1 and 2.

\begin{figure}
  \centering
  \begin{minipage}{0.32\linewidth}\label{F1}
  \centering
  \includegraphics[width=0.55\textwidth]{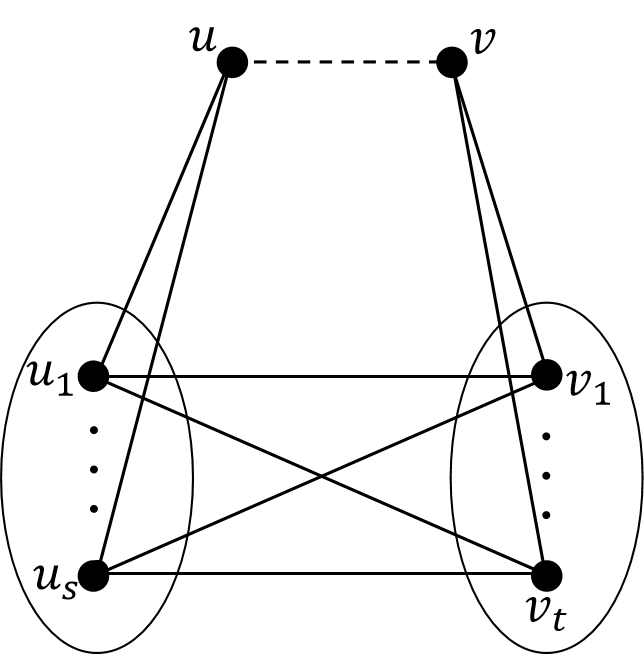}
  \caption*{\footnotesize{$\dot{G}(s,t)$}}
  \end{minipage}
  \begin{minipage}{0.32\linewidth}\label{F1}
  \centering
  \includegraphics[width=0.55\textwidth]{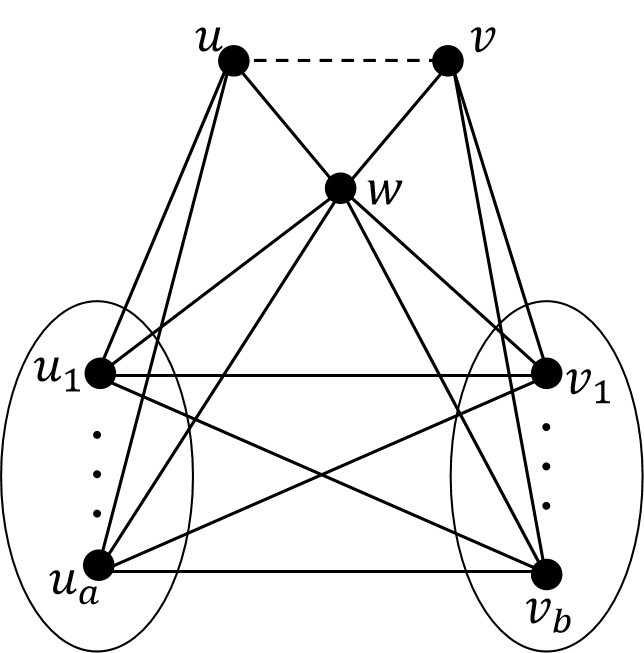}
  \caption*{\footnotesize{$\dot{G_1}(a,b)$}}
  \end{minipage}
 \begin{minipage}{0.32\linewidth}\label{F2}
  \centering
  \includegraphics[width=0.55\textwidth]{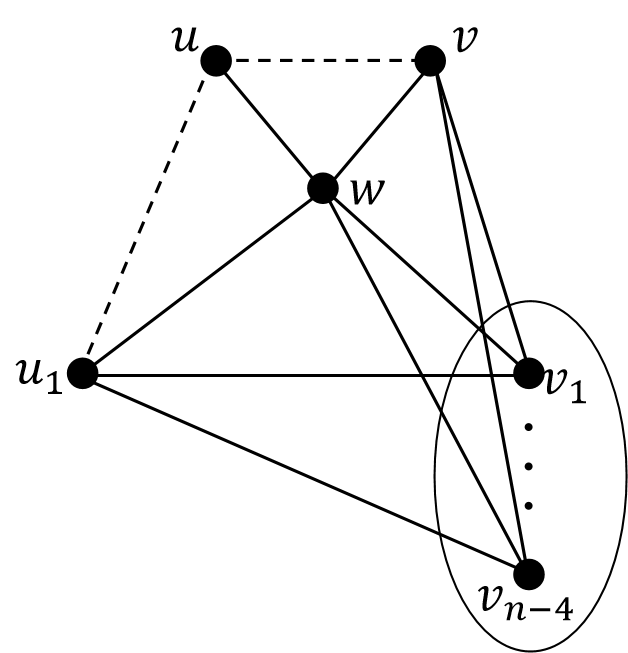}
  \caption*{\footnotesize{$\dot{G_1'}(1,n-4)$}}
  \end{minipage}
  \caption{\centering{The signed graphs $\dot{G}(s,t)$, $\dot{G_1}(a,b)$, and $\dot{G_1'}(1,n-4)$}}
\end{figure}

\begin{theorem}\label{c3edge}\cite{WHL}
Let $\dot{G}=(G,\sigma)$ be a connected $\mathcal{K}_3^-$-free unbalanced signed graph of order $n$. Then
$$e(\dot{G})\leqslant \frac{n(n-1)}{2}-(n-2),$$
with equality holding if and only if $\dot{G}\sim \dot{G}(s,t)$, where $s+t=n-2$ and $s$, $t\geqslant1$ (see Fig. 1).
\end{theorem}

\begin{theorem}\label{c3spectrum}\cite{WHL}
Let $\dot{G}=(G,\sigma)$ be a connected $\mathcal{K}_3^-$-free unbalanced signed graph of order $n$. Then
$$\rho(\dot{G})\leqslant\frac{1}{2}(\sqrt{n^2-8}+n-4),$$
with equality holding if and only if $\dot{G}\sim\dot{G}(1,n-3)$.
\end{theorem}

Let $\mathcal{K}_4^-$ be the set of the unbalanced $K_4$. Up to switching equivalence, we have $\mathcal{K}_4^-=\{\dot{H_1},\dot{H_2}\}$, where $\dot{H_1}$ is the signed $K_4$ with exactly one negative edge and $\dot{H_2}$ is the signed $K_4$ with two independent negative edges. If $\dot{G}$ is a $\mathcal{K}_4^-$-free signed graph of order $n$ with maximum edges, then $e(\dot{G})=\frac{n(n-1)}{2}$ with equality holding if and only if $\dot{G}\sim(K_n,+)$. Therefore, we would focus the attention on $\mathcal{K}_4^-$-free unbalanced signed graphs. Let
$$\mathcal{G}=\big\{\dot{G_1}(a,b),\, \dot{G_1'}(1,n-4), \, \dot{G_2}(c,d), \, \dot{G_3}(1,n-5), \, \dot{G_4}(1,n-5), \, \dot{G_5}(1,n-5)\big\},$$
where $a+b=n-3$, $a$, $b\geqslant 0$, and $c+d=n-4$, $c$, $d\geqslant 1$ (see Fig. 1 and 2). For any signed graph $\dot{G}\in\mathcal{G}$, we have $\dot{G}$ is $\mathcal{K}_4^-$-free unbalanced, and
\begin{equation}\label{n-3}
e(\dot{G})=\frac{n(n-1)}{2}-(n-3).
\end{equation}

If $\dot{G}$ is a $\mathcal{K}_4^-$-free signed graph of order $n$ with maximum spectral radius, then $\rho(\dot{G})=n-1$ with equality holding if and only if $\dot{G}\sim(K_n,+)$. Therefore, we would focus the attention on $\mathcal{K}_4^-$-free unbalanced signed graphs. In Section 2, for any signed graph $\dot{G}\in\mathcal{G}$, we will prove that
$\lambda_1(\dot{G})\leqslant n-2,$
with equality holding if and only if $\dot{G}\sim \dot{G_1}(0,n-3)$.

Among all unbalanced signed graph, Tur$\rm\acute{a}$n number of $\mathcal{K}_4^-$ will be determined in Theorem \ref{edge}, and spectral Tur$\rm\acute{a}$n number of $\mathcal{K}_4^-$ will be determined in Theorem \ref{spectrum}. Their proofs will be presented in Section 3 and Section 4.

\begin{theorem}\label{edge}
Let $\dot{G}=(G,\sigma)$ be a $\mathcal{K}_4^-$-free unbalanced signed graph of order $n$ ($n\geqslant7$). Then
$$e(\dot{G})\leqslant \frac{n(n-1)}{2}-(n-3),$$
with equality holding if and only if $\dot{G}$ is switching equivalent to a signed graph in $\mathcal{G}$.
\end{theorem}

\begin{figure}
  \centering
  \begin{minipage}{0.24\linewidth}\label{F3}
  \centering
  \includegraphics[width=0.75\textwidth]{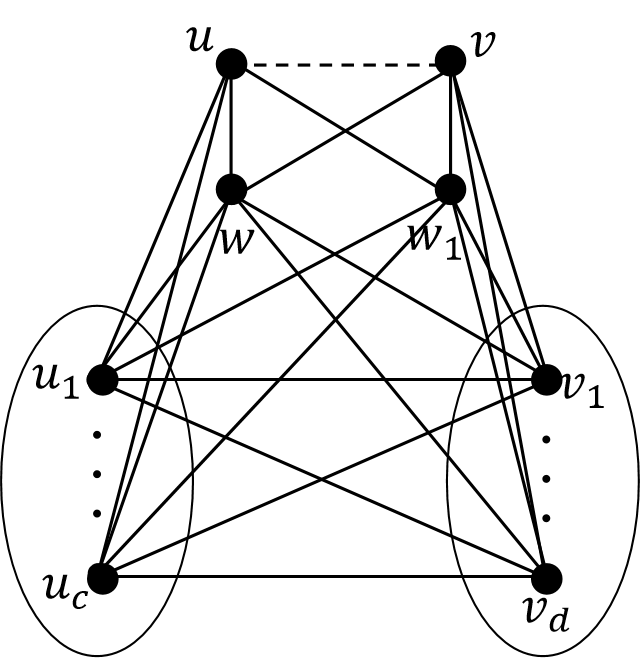}
  \caption*{\footnotesize{$\dot{G_2}(c,d)$}}
  \end{minipage}
  \begin{minipage}{0.24\linewidth}\label{F4}
  \centering
  \includegraphics[width=0.75\textwidth]{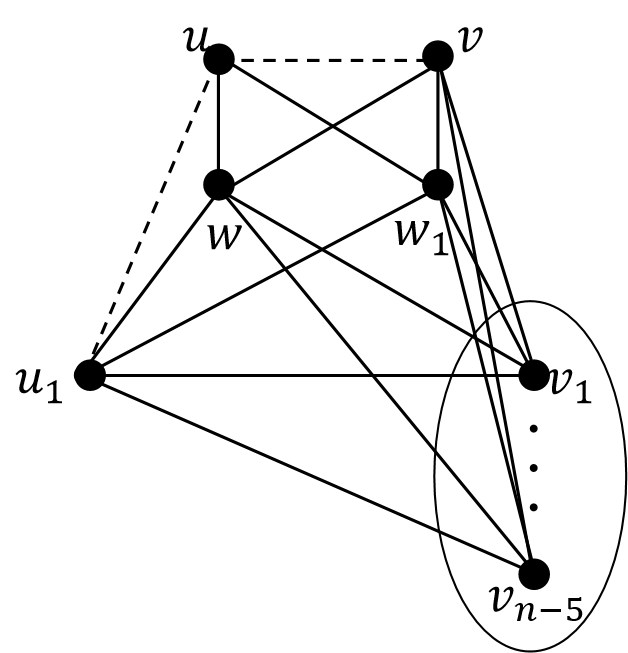}
  \caption*{\footnotesize{$\dot{G_3}(1,n-5)$}}
  \end{minipage}
 \begin{minipage}{0.24\linewidth}\label{F5}
  \centering
  \includegraphics[width=0.75\textwidth]{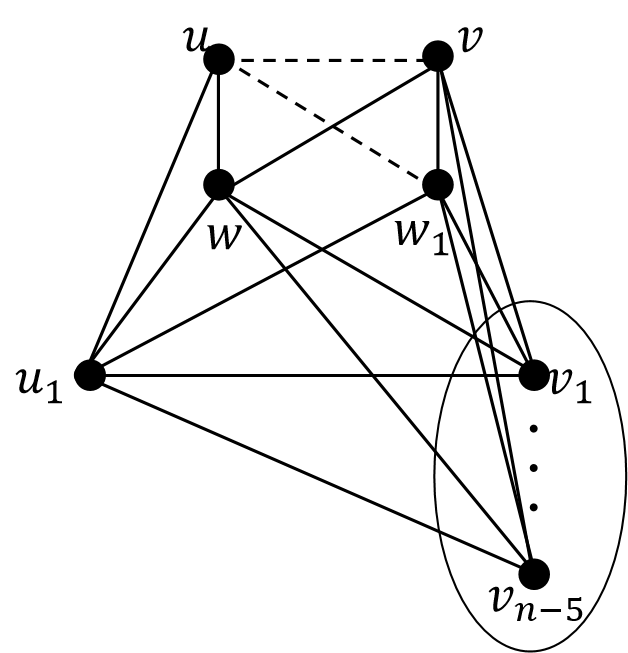}
  \caption*{\footnotesize{$\dot{G_4}(1,n-5)$}}
  \end{minipage}
 \begin{minipage}{0.24\linewidth}\label{F6}
  \centering
  \includegraphics[width=0.75\textwidth]{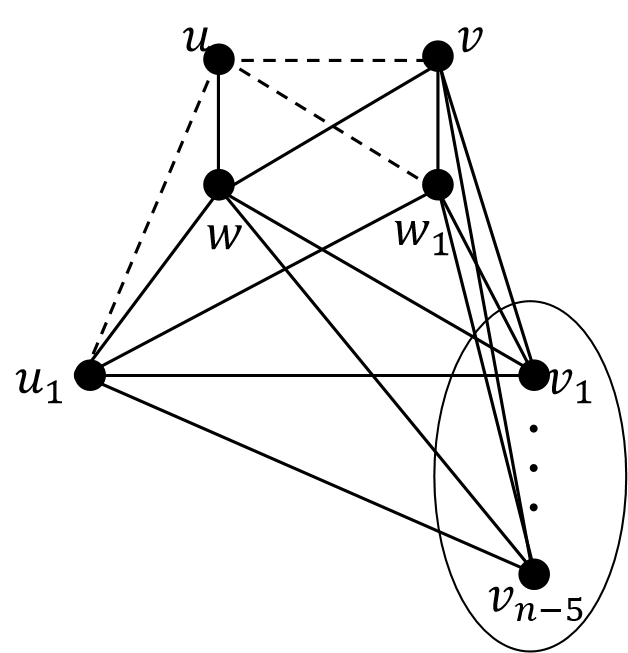}
  \caption*{\footnotesize{$\dot{G_5}(1,n-5)$}}
  \end{minipage}
  \caption{The signed graphs $\dot{G_2}(c,d)$, $\dot{G_3}(1,n-5)$, $\dot{G_4}(1,n-5)$, and $\dot{G_5}(1,n-5)$.}
\end{figure}

\begin{theorem}\label{spectrum}
Let $\dot{G}=(G,\sigma)$ be a $\mathcal{K}_4^-$-free unbalanced signed graph of order $n$. Then
$$\rho(\dot{G})\leqslant n-2,$$
with equality holding if and only if $\dot{G}\sim \dot{G_1}(0,n-3)$.
\end{theorem}

\section{The indices of signed graphs in $\mathcal{G}$}
For any signed graph $\dot{G}$ in $\mathcal{G}$, we will show that $\lambda_1(\dot{G})\leqslant n-2$, with equality holding if and only if $\dot{G}\sim\dot{G_1}(0,n-3)$. The equitable quotient matrix technique and Cauchy Interlacing Theorem are two main tools in our proof.

Let $M$ be a real symmetric matrix with the following block form
$$M=\left(
      \begin{array}{ccc}
        M_{11} & \cdots & M_{1m} \\
        \vdots & \ddots & \vdots \\
        M_{m1} & \cdots & M_{mm} \\
      \end{array}
    \right).
$$
For $1\leqslant i,j\leqslant m$, let $q_{ij}$ denote the average row sum of $M_{ij}$. The matrix $Q=(q_{ij})$ is called the quotient matrix of $M$. Moreover, if for each pair $i,j$, $M_{ij}$ has a constant row sum, then $Q$ is called a equitable quotient matrix of $M$.
\begin{lemma}\label{equitable}\cite{BH}
Let $Q$ be an equitable quotient matrix of matrix $M$. Then the matrix $M$ has the following two kinds of eigenvalues.

(1) The eigenvalues coincide with the eigenvalues of $Q$.

(2) The eigenvalues of $M$ not in {\rm Spec}($Q$) remain unchanged if some scalar multiple of the all-one block $J$ is added to block $M_{ij}$ for each $1\leqslant i,j\leqslant m$.

Furthermore, if $M$ is nonnegative and irreducible, then $\lambda_1(M)=\lambda_1(Q).$
\end{lemma}

\begin{lemma}\label{interlacing}\cite{BH}
Let $A$ be a symmetric matrix of order $n$ with eigenvalues $\lambda_1\geqslant \lambda_2\geqslant\cdots\geqslant\lambda_n$ and $B$ be a principal submatrix of $A$ of order $m$ with eigenvalues $\mu_1\geqslant\mu_2\geqslant\cdots\geqslant\mu_m$. Then the eigenvalues of $B$ interlace the eigenvalues of $A$, that is, $\lambda_i\geqslant\mu_i\geqslant\lambda_{n-m+i}$ for $i=1,\cdots,m$.
\end{lemma}

The clique number of a graph $G$, denoted by $\omega(G)$, is the maximum order of a clique in $G$. The balanced clique number of a signed graph $\dot{G}$, denoted by $\omega_b(\dot{G})$, is the maximum order of a  balanced clique in $\dot{G}$.

\begin{lemma}\label{lower}\cite{W} Let $G$ be a graph of order $n$. Then
$$\lambda_1(G)\leqslant\left(1-\frac{1}{\omega(G)}\right)n.$$
\end{lemma}

\begin{lemma}\label{balanced}\cite{WYQ}
Let $\dot{G}$ be a signed graph of order $n$. Then
$$\lambda_1(\dot{G})\leqslant\left(1-\frac{1}{\omega_b(\dot{G})}\right)n.$$
\end{lemma}

\begin{lemma}\cite{S}\label{spanning}
Every signed graph $\dot{G}$ contains a balanced spanning subgraph, say $\dot{H}$, which satisfies $\lambda_1(\dot{G})\leqslant\lambda_1(\dot{H})$.
\end{lemma}

\begin{remark}\cite{WHL}\label{spanning2}
There is a switching equivalent graph $\dot{G}_U$ such that eigenvector $\bm{x}$ of $A(\dot{G}_{U})$ corresponding to $\lambda_1(\dot{G})$ is non-negative. By the proof of \cite[Theorem 3.1]{S}, the balanced spanning subgraph $\dot{H}$ in Lemma \ref{spanning} may be obtained from $\dot{G}_U$ by removing all negative edges.
\end{remark}

\begin{lemma}\label{unbalanced}
If $\dot{G}$ is a connected signed graph, then $\lambda_1(\dot{G})\leqslant\lambda_1(G)$. Moreover, the equality holds if and only if $\dot{G}$ is balanced.
\end{lemma}
\begin{proof}
Let $\dot{G}_U$ be a signed graph defined in Remark \ref{spanning2}, and $\dot{H}$ be a spanning subgraph of $\dot{G}_{U}$ by removing all negative edges and $\lambda_1(\dot{G})\leqslant\lambda_1(\dot{H})$. Thus $A(\dot{H})$ is a nonnegative matrix, $A(G)$ is a nonnegative irreducible matrix, and $\dot{H}$ is a subgraph of $G$. By Perron-Frobenius Theorem, we know that $\lambda_1(\dot{H})\leqslant \lambda_1(G)$. So, $\lambda_1(\dot{G})\leqslant\lambda_1(G)$.

If $\lambda_1(\dot{G})=\lambda_1(G)$, then $\lambda_1(\dot{H})= \lambda_1(G)$. By Perron-Frobenius Theorem, we know that $A(\dot{H})=A(G)$ and then $\dot{H}=G$ and $\dot{G}_U=(G,+)$, so $\dot{G}$ is balanced. If $\dot{G}$ is balanced, then $\dot{G}\sim(G,+)$, and then $\lambda_1(\dot{G})=\lambda_1(G)$.
\end{proof}

Recall that
$$\mathcal{G}=\big\{\dot{G_1}(a,b),\, \dot{G_1'}(1,n-4), \, \dot{G_2}(c,d), \, \dot{G_3}(1,n-5), \, \dot{G_4}(1,n-5), \, \dot{G_5}(1,n-5)\big\},$$
where $a+b=n-3$, $a$, $b\geqslant 0$, and $c+d=n-4$, $c$, $d\geqslant 1$.
\begin{lemma}\label{largest}
For any graph $\dot{G}$ in $\mathcal{G}$, we have $\lambda_1(\dot{G})\leqslant n-2$, with equality holding if and only if $\dot{G}\sim \dot{G_1}(0,n-3)$.
\end{lemma}
\begin{proof}[\rm{\textbf{Proof.}}]
We will complete the proof by showing the following five claims. Firstly, we claim that $\lambda_1(\dot{G_1}(0,n-3))=n-2$.

\begin{claim}
$\lambda_1(\dot{G_1}(0,n-3))=n-2$.
\end{claim}
\begin{proof}[\rm{\textbf{Proof of Claim 1.}}]
For the signed graph $\dot{G_1}(0,n-3)$, we give a vertex partition with $V_1=\{u\}$, $V_2=\{v\}$, $V_3=\{w\}$ and $V_4=V(\dot{G})\setminus\{u,v,w\}$. Then the adjacency matrix $A(\dot{G_1}(0,n-3))$ and its corresponding equitable quotient matrix $Q_1$ are as following
$$A(\dot{G_1}(0,n-3))={\footnotesize\begin{bmatrix}
        0 & -1 & 1 & \bm{0^T} \\
        -1 & 0 & 1 & \bm{j^T_{n-3}} \\
        1 & 1 & 0 & \bm{j^T_{n-3}} \\
        \bm{0} & \bm{j_{n-3}} & \bm{j_{n-3}} & (J-I)_{n-3}\\
   \end{bmatrix}}\ \text{and}\
Q_1={\footnotesize\begin{bmatrix}
        0 & -1 & 1 & 0 \\
        -1 & 0 & 1 & n-3  \\
        1 & 1 & 0 & n-3  \\
        0 & 1 & 1 & n-4 \\
   \end{bmatrix}.}
$$
By Lemma \ref{equitable} (1), the eigenvalues of $Q_1$ are also the eigenvalues of $A(\dot{G_1}(0,n-3))$. The characteristic polynomial of $Q_1$ is
\begin{align*}
P_{Q_1}(x)=(x-n+2)(x-1)(x+1)(x+2).
\end{align*}
Hence, $\lambda_1(Q_1)=n-2$. Add some scalar multiple of the all-one block $J$ to block of $A(\dot{G_1}(0,n-3))$ and $A(\dot{G_1}(0,n-3))$ becomes
$$A_1={\footnotesize\begin{bmatrix}
 0 & 0 & 0 & \bm{0^T} \\
 0 & 0 & 0 & \bm{0^T} \\
 0 & 0 & 0 & \bm{0^T} \\
 \bm{0} & \bm{0} & \bm{0} & -I_{n-3} \\
   \end{bmatrix}.}\
$$
By Lemma \ref{equitable} (2), there are $n-4$ eigenvalues of $\dot{G_1}(0,n-3)$ contained in the spectra of $A_1$. Since Spec($A_1$)=$\{-1^{[n-3]},0^{[3]}\}$, we have $\lambda_1(\dot{G_1}(0,n-3))=\lambda_1(Q_1)=n-2$.
\end{proof}
Next, we claim that $\lambda_1(\dot{G})< n-2$ for any graph $\dot{G}\in \mathcal{G}\setminus\{ \dot{G_1}(0,n-3)\}$.
\begin{claim}
$\lambda_1(\dot{G_1}(a,b))<n-2,\ 1\leqslant a\leqslant b$.
\end{claim}
\begin{proof}[\rm{\textbf{Proof of Claim 2.}}]
We prove that Claim by showing
$$\lambda_1(\dot{G_1}(\lfloor\frac{n-3}{2}\rfloor,\lceil\frac{n-3}{2}\rceil))<\cdots<\lambda_1(\dot{G_1}(1,n-4))<n-2.$$
Partition the vertices set of $\dot{G_1}(a,b)$ as $V_1=\{u\}$, $V_2=\{v\}$, $V_3=\{w\}$, $V_4=N(u)\setminus\{v,w\}$ and $V_5=N(v)\setminus\{u,w\}$. Then the adjacency matrix $A(\dot{G_1}(a,b))$ and its corresponding equitable quotient matrix $Q_2(a,b)$ are as following
$$A(\dot{G_1}(a,b))={\footnotesize\begin{bmatrix}
        0 & -1 & 1 & \bm{j^T_{a}} & \bm{0^T} \\
        -1 & 0 & 1 & \bm{0^T} & \bm{j^T_{b}} \\
        1 & 1 & 0 & \bm{j^T_{a}} & \bm{j^T_{b}} \\
        \bm{j_{a}} & \bm{0} & \bm{j_{a}} & (J-I)_{a} & J_{b} \\
        \bm{0} & \bm{j_{b}} & \bm{j_{b}} & J_{a} & (J-I)_{b} \\
   \end{bmatrix}}\ \text{and}\
Q_2(a,b)={\footnotesize\begin{bmatrix}
        0 & -1 & 1 & a & 0 \\
        -1 & 0 & 1 & 0 & b \\
        1 & 1 & 0 & a & b \\
        1 & 0 & 1 & a-1 & b \\
        0 & 1 & 1 & a & b-1 \\
   \end{bmatrix}.}
$$
By Lemma \ref{equitable} (1), the eigenvalues of $Q_2(a,b)$ are also the eigenvalues of $A(\dot{G_1}(a,b))$. The characteristic polynomial of $Q_2(a,b)$ is
\begin{align}\label{Q2}
P_{Q_2(a,b)}(x,a,b)=&x^5-(a+b-2)x^4-(3a+3b+2)x^3+(2ab-a-b-4)x^2\notag\\
&+(5ab+3a+3b)x+2ab+2a+2b+2.
\end{align}
Add some scalar multiple of the all-one block $J$ to block of $A(\dot{G_1}(a,b))$ and $A(\dot{G_1}(a,b))$ becomes
$$A_2={\footnotesize\begin{bmatrix}
        0 & 0 & 0 & \bm{0^T} & \bm{0^T} \\
        0 & 0 & 0 & \bm{0^T} & \bm{0^T} \\
        0 & 0 & 0 & \bm{0^T} & \bm{0^T} \\
        \bm{0} & \bm{0} & \bm{0} & -I_{a} & 0 \\
        \bm{0} & \bm{0} & \bm{0} & 0 & -I_{b} \\
   \end{bmatrix}.}\
$$
By Lemma \ref{equitable} (2), there are $n-5$ eigenvalues of $\dot{G_1}(a,b)$ contained in the spectra of $A_2$. Since $\lambda_1(Q_2(a,b))>0$ and Spec($A_2$)=$\{-1^{[n-3]},0^{[3]}\}$, we have $\lambda_1(\dot{G_1}(a,b))=\lambda_1(Q_2(a,b))$.

Noting that
$$P_{Q_2(a,b)}(x,a,b)-P_{Q_2(a-1,b+1)}(x,a-1,b+1)=(b-a+1)(2x+1)(x+2),$$
then $P_{Q_2(a,b)}(x,a,b)>P_{Q_2(a-1,b+1)}(x,a-1,b+1)$ when $x>-\frac{1}{2}$. Hence, $\lambda_1(Q_2(a,b))<\lambda_1(Q_2(a-1,b+1))$. Thus, $\lambda_1(\dot{G_1}(a,b))<\lambda_1(\dot{G_1}(a-1,b+1))$, and then
$$\lambda_1(\dot{G_1}(\lfloor\frac{n-3}{2}\rfloor,\lceil\frac{n-3}{2}\rceil))<\cdots<\lambda_1(\dot{G_1}(1,n-4)).$$
Now we will prove that $\lambda_1(\dot{G_1}(1,n-4))<n-2$. By (\ref{Q2}), we have
\begin{align*}
P_{Q_2(1,n-4)}(x,1,n-4)&=(x+2)g_2(x),
\end{align*}
where $g_2(x)=x^4-(n-3)x^3-(n-1)x^2+(3n-11)x+2n-6$. Note that, for $n\geqslant 7$,
$$g_2(n-2)=2n^2-11n+12>0.$$
To complete the proof, it suffices to prove that $\lambda_2(\dot{G_1}(1,n-4))<n-2$. In fact, by Lemmas \ref{interlacing} and \ref{balanced}, we have
\begin{align*}
\lambda_2(\dot{G_1}(1,n-4))&\leqslant \lambda_1(\dot{G_1}(1,n-4)-v)\\
&\leqslant\left(1-\frac{1}{\omega_b(\dot{G_1}(1,n-4)-v)}\right)(n-1)\\
&=\frac{(n-1)(n-3)}{n-2}<n-2.
\end{align*}
Hence, $\lambda_1(\dot{G_1}(1,n-4))<n-2$.
\end{proof}

\begin{claim}
$\lambda_1(\dot{G_1'}(1,n-4))<n-2.$
\end{claim}
\begin{proof}[\rm{\textbf{Proof of Claim 3.}}]
Partition the vertices set of $\dot{G_1'}(1,n-4)$ as $V_1=\{u\}$, $V_2=\{v\}$, $V_3=\{w\}$, $V_4=\{u_1\}$ and $V_5=N(v)\setminus\{u,w\}$. The corresponding equitable quotient matrix $Q_3$ of $A(\dot{G_1'}(1,n-4))$ is as following
$$Q_3={\footnotesize\begin{bmatrix}
        0 & -1 & 1 & -1 & 0 \\
        -1 & 0 & 1 & 0 & n-4 \\
        1 & 1 & 0 & 1 & n-4 \\
        -1 & 0 & 1 & 0 & n-4 \\
        0 & 1 & 1 & 1 & n-5 \\
   \end{bmatrix}.}
$$
Furthermore, by Lemma \ref{equitable}, we know that the $\lambda_1(\dot{G_1'}(1,n-4))=\lambda_1(Q_3)$. The characteristic polynomial of $Q_3$ is $P_{Q_3}(x)=xg_3(x),$
where $g_3(x)=x^4 + (5 - n)x^3 + (7 - 3n)x^2 + (n - 5)x + (4n - 12)$.
Note that, for $n\geqslant 7$,
$$g_3(n-2)=2n^2-7n+2>0.$$
To complete the proof, it suffices to prove that $\lambda_2(\dot{G_1'}(1,n-4))<n-2$. In fact, by Lemmas \ref{interlacing} and \ref{balanced}, we have
\begin{align*}
\lambda_2(\dot{G_1'}(1,n-4))&\leqslant \lambda_1(\dot{G_1'}(1,n-4)-v)\\
&\leqslant\left(1-\frac{1}{\omega_b(\dot{G_1'}(1,n-4)-v)}\right)(n-1)\\
&=\frac{(n-1)(n-3)}{n-2}<n-2.
\end{align*}
Hence, $\lambda_1(\dot{G_1'}(1,n-4))<n-2$.
\end{proof}

\begin{claim}
$\lambda_1(\dot{G_2}(c,d))<n-2,\ 1\leqslant c\leqslant d$.
\end{claim}
\begin{proof}[\rm{\textbf{Proof of Claim 4.}}]
We may consider the underlying graph $G_2(c,d)$ for $\dot{G_2}(c,d)$ and we prove that Claim by showing
$$\lambda_1(G_2(\lfloor\frac{n-4}{2}\rfloor,\lceil\frac{n-4}{2}\rceil))<\cdots<\lambda_1(G_2(1,n-5))<n-2.$$
Partition the vertices set of $G_2(c,d)$ as $V_1=\{u\}$, $V_2=\{v\}$, $V_3=\{w,w_1\}$, $V_4=N(u)\setminus\{v,w,w_1\}$ and $V_5=N(v)\setminus\{u,w,w_1\}$. Then the corresponding equitable quotient matrix $Q_4(c,d)$ of $A(G_2(c,d))$ is
$$Q_4(c,d)={\footnotesize\begin{bmatrix}
        0 & 1 & 2 & c & 0 \\
        1 & 0 & 2 & 0 & d \\
        1 & 1 & 0 & c & d \\
        1 & 0 & 2 & c-1 & d \\
        0 & 1 & 2 & c & d-1 \\
   \end{bmatrix}.}$$
Since $A(G_2(c,d))$ is nonnegative and irreducible, we have $\lambda_1(A(G_2(c,d)))=\lambda_1(Q_4(c,d))$ by Lemma \ref{equitable}. The characteristic polynomial of $Q_4(c,d)$ is
\begin{align}\label{Q3}
P_{Q_4(c,d)}(x,c,d)=&x^5+(2-c-d)x^4-(4c+4d+4)x^3+(2cd-2c-2d-14)x^2\notag\\
&+(3cd+5c+5d-13)x+4c+4d-4cd-4.
\end{align}
Noting that
$$P_{Q_4(c,d)}(x,c,d)-P_{Q_4(c-1,d+1)}(x,c-1,d+1)=(d-c+1)(2x^2+3x-4),$$
then $P_{Q_4(c-1,d+1)}(x,c-1,d+1)<P_{Q_4(c,d)}(x,c,d)$ when $x> 1$. Hence, $\lambda_1(Q_4(c,d))<\lambda_1(Q_4(c-1,d+1))$. Thus, $\lambda_1(G_2(c,d))<\lambda_1(G_2(c-1,d+1))$ and,
$$\lambda_1(G_2(\lfloor\frac{n-4}{2}\rfloor,\lceil\frac{n-4}{2}\rceil))<\cdots<\lambda_1(G_2(1,n-5)).$$
Noting that, by ($\ref{Q3}$)
\begin{align*}
P_{Q_4(1,n-5)}(x,1,n-5)=x^5+(6-n)x^4+(12-4n)x^3-16x^2+(8n-48)x,
\end{align*}
then, for $n\geqslant 7$, we have
$$P_{Q_4(1,n-5)}(n-2,1,n-5)=4n(n-2)(n-6)>0.$$
To complete the proof, it suffices to prove that $\lambda_2(G_2(1,n-5))<n-2$. In fact, by Lemmas \ref{interlacing} and \ref{lower}, we have
\begin{align*}
\lambda_2(G_2(1,n-5))&\leqslant \lambda_1(G_2(1,n-5)-v)\\
&\leqslant\left(1-\frac{1}{\omega(G_2(1,n-5)-v)}\right)(n-1)\\
&=\frac{(n-1)(n-4)}{n-3}<n-2.
\end{align*}
Therefore, we have
$\lambda_1(G_2(1,n-5))<n-2.$ By Lemma \ref{unbalanced}, we have $\lambda_1(\dot{G_2}(c,d))<\lambda_1(G_2(c,d))<n-2$.
\end{proof}
\begin{claim}
$\lambda_1(\dot{G_i}(1,n-5))<n-2, \ i=3,4,5.$
\end{claim}
\begin{proof}[\rm{\textbf{Proof of Claim 5.}}]
Write $G_i(1,n-5)$ as the underlying graph of $\dot{G_i}(1,n-5)$ for $i=3,4,5$. Noting that $G_2(1,n-5)\cong G_i(1,n-5)$ for $i=3,4,5$. Hence, by Lemma \ref{unbalanced} we have $\lambda_1(\dot{G_i}(1,n-5))<\lambda_1(G_i(1,n-5))=\lambda_1(G_2(1,n-5))<n-2$.
\end{proof}
\end{proof}

\section{A proof of Tur$\mathbf{\acute{a}}$n number of $\mathcal{K}_4^-$}
Let $\dot{G}$ be an unbalanced signed graph of order $n$ ($n\geqslant7$). For any vertex $v$ in $V(\dot{G})$, $N_{\dot{G}}(v)$ (or $N(v)$) is the set of the neighbors of $v$ and $N_{\dot{G}}[v]=N_{\dot{G}}(v)\cup\{v\}$ (or $N[v]$). For $U\subseteq V(G)$, let $\dot{G}[U]$ be the subgraph induced by $U$.
Let $\dot{G}$ be a $\mathcal{K}_4^-$-free unbalanced signed graph with maximum edges. In fact, $\dot{G}$ is connected. Otherwise, for some two vertices $u$ and $v$ in distinct components, we add the edge $uv$ to $\dot{G}$. Then $\dot{G}+uv$ is a $\mathcal{K}_4^-$-free unbalanced signed graph with more edges than $\dot{G}$, which is a contradiction.
\begin{proof}[\rm{\textbf{Proof of Theorem \ref{edge}.}}]
Let $\dot{G}$ be a $\mathcal{K}_4^-$-free unbalanced signed graph with maximum edges. Then $\dot{G}$ contains at least one negative cycle, and assume the smallest length of the negative cycles is $\ell$. Since each signed graph in $\mathcal{G}$ is $\mathcal{K}_4^-$-free and unbalanced, from (\ref{n-3}) we have
\begin{equation}\label{medge}
e(\dot{G})\geqslant \frac{n(n-1)}{2}-(n-3).
\end{equation}
If $\ell\geqslant 4$, then $\dot{G}$ is $\mathcal{K}_3^-$-free. Noting that $\dot{G}$ is connected, by Theorem \ref{c3edge}, we have
$$e(\dot{G})\leqslant \frac{n(n-1)}{2}-(n-2)<\frac{n(n-1)}{2}-(n-3),$$
which is a contradiction to (\ref{medge}).
Hence, assume that $uvwu$ is an unbalanced $K_3$ with $\sigma(uv)=-1$ and $\sigma(uw)=\sigma(vw)=+1$. Suppose that $e(\dot{G})=\frac{n(n-1)}{2}-q$. Since $e(\dot{G})\geqslant\frac{n(n-1)}{2}-(n-3)$, we have $q\leqslant n-3$. Now the proof will be divided into two cases.

\textbf{Case 1.} $|N(u)\cap N(v)|=1$.

Let $N(u)\setminus\{v,w\}=\{u_1,\cdots,u_{a}\}$ and $N(v)\setminus\{u,w\}=\{v_1,\cdots,v_{b}\}$. Then, $a+b\leqslant n-3$. In this case, we have
\begin{align*}
e(\dot{G})&=e(\dot{G}[V(\dot{G})\setminus\{u,v\}])+e(\dot{G}[\{u,v\},V(\dot{G})\setminus\{u,v\}])+1\\
&\leqslant \frac{(n-2)(n-3)}{2}+(a+1)+(b+1)+1\\
&\leqslant \frac{n(n-1)}{2}-(n-3).
\end{align*}
Hence, by (\ref{medge}) we have $e(\dot{G})=\frac{n(n-1)}{2}-(n-3).$
Furthermore, $\dot{G}[V(\dot{G})\setminus\{u,v\}]$ is a clique and $a+b=n-3$, namely each vertex in $V(\dot{G})\setminus\{u,v,w\}$ is adjacent to $u$ or $v$. Since the switching operation remains the sign of cycles, any $\dot{G}_U$ is still $\mathcal{K}^-_4$-free and unbalanced. We may suppose $\sigma(uu_i)=+1$ for any $1\leqslant i\leqslant a$. Otherwise, we do switching operation at some $u_i$. Similarly, suppose $\sigma(vv_j)=+1$ for any $1\leqslant j\leqslant b$.

Without loss of generality, assume that $a\leqslant b$. The assumption $n\geqslant 7$ ensures that $b\geqslant 2$. For any two vertices, say $v_i$ and $v_j$, in $N(v)\setminus\{u,w\}$, we have $\dot{G}[\{v,w,v_i,v_j\}]\sim (K_4,+)$. Hence, $\sigma(wv_i)=\sigma(v_iv_j)=+1$ for $1\leqslant i,j\leqslant b$, and then $\dot{G}[N(v)\setminus\{u\}]$ is a clique with all positive edges.

If $a=0$, then $\dot{G}\sim \dot{G_1}(0,n-3)$. If $a=1$, then $\dot{G}[\{w,u_1,v_{i},v_{j}\}]\sim (K_4,+)$ and $\sigma(u_1w)=\sigma(u_1v_{i})$ for $1\leqslant i\leqslant b$. If $\sigma(u_1w)=\sigma(u_1v_{i})=+1$, then $\dot{G}\sim \dot{G_1}(1,n-4)$. If $\sigma(u_1w)=\sigma(u_1v_{i})=-1$, then $\dot{G}\sim \dot{G_1'}(1,n-4)$.

If $a\geqslant 2$, for any two vertices, say $u_{i}$ and $u_{j}$, in $N(u)\setminus\{v,w\}$, we have $\dot{G}[\{u,w,u_{i},u_{j}\}]\sim (K_4,+)$. Hence, $\dot{G}[N(u)\setminus\{v\}]$ is a clique with all positive edges. Noting that $\dot{G}[\{w,u_i,v_j,v_t\}]$ $\sim (K_4,+)$, we have $\sigma(u_i v_j)=+1$ for $1\leqslant i\leqslant a$ and $1\leqslant j\leqslant b$. Therefore, $\dot{G}\sim \dot{G}_1(a, b)$ with $a\geqslant 2$.

Hence, in this case, $\dot{G}=\dot{G_1}(a,b)$ with $a$, $b\geqslant 0$, $a+b=n-3$, or $\dot{G}=\dot{G_1'}(1,n-4)$.

\textbf{Case 2.} $|N(u)\cap N(v)|\geqslant 2$.

Let $N(u)\cap N(v)=\{w,w_1,\cdots,w_k\}$, $N[u]\setminus N[v]=\{u_1,\cdots,u_{c}\}$, and $N[v]\setminus N[u]=\{v_1,\cdots,v_{d}\}$. For any $1\leqslant i\leqslant k$, we claim that $w$ is not adjacent to $w_i$. Otherwise, suppose that $w$ is adjacent to $w_1$. Since $\dot{G}[\{u,v,w\}]$ is an unbalanced $K_3$, we have $\dot{G}[\{u,v,w,w_1\}]$ is an unbalanced $K_4$, which is a contradiction to the assumption that $\dot{G}$ is $\mathcal{K}_4^-$-free and unbalanced.

We further claim that each vertex of $\dot{G}$ is adjacent to the vertex $u$ or $v$. Otherwise suppose there are $r$ vertices in $V(\dot{G})\setminus(N(u)\cup N(v))$. Then $3+c+d+k+r=n$ holds. From the inequality
$$n-3\geqslant q\geqslant c+d+k+2r=n-3+r,$$
we have $r=0$. Furthermore $q=c+d+k=n-3$. Then $\dot{G}[V(\dot{G})\setminus\{u,v,w\}]$ is a clique and $w$ is adjacent to $u_i$ and $v_j$ for $1\leqslant i\leqslant c$ and $1\leqslant j\leqslant d$.

We may suppose $\sigma(uu_i)=\sigma(vv_j)=\sigma(vw_t)=+1$ for $1\leqslant i\leqslant c$, $1\leqslant j\leqslant d$, and $1\leqslant t\leqslant k$. Without loss of generality, assume that $0\leqslant c\leqslant d$. If $c=0$, then we may set $\dot{G}^*=\dot{G}_{\{u\}}$, namely, $\dot{G}^*$ is obtained from $\dot{G}$ by a switching operation at the vertex $u$. Then in $\dot{G}^*$, we have $|N(u)\cap N(w)|=1$ and $\sigma(u w)=-1$. Thus $\dot{G}^*$ satisfies the condition of Case 1. Then we have
$$e(\dot{G})=e(\dot{G}^*)=\frac{n(n-1)}{2}-(n-3),$$
and $\dot{G}\sim\dot{G}^*\sim \dot{G_1}(a,b)$ or $ \dot{G_1'}(1,n-4)$. Now suppose $c\geqslant 1$, and we will distinguish two subcases.

\textbf{Subcase 2.1.} $|N(u)\cap N(v)|=2$.

Since $k=1$ and $n\geqslant 7$, we have $d\geqslant 2$. For any two vertices, say $v_i$ and $v_j$, in $N[v]\setminus N[u]$, we have $\dot{G}[\{v,w,v_i,v_j\}]\sim (K_4,+)$. Hence, $\sigma(wv_i)=+1$ for $1\leqslant i\leqslant d$ and $\dot{G}[N[v]\setminus N[u]]$ is a clique with all positive edges. Similarly, we have $\sigma(w_1v_i)=+1$ for $1\leqslant i\leqslant d$.

If $c=1$, then $d=n-5$. Since the clique $\dot{G}[\{w,u_1,v_i,v_j\}]\sim (K_4,+)$, we have $\sigma(u_1w)=\sigma(u_1v_i)$, and similarly, we get $\sigma(u_1w_1)=\sigma(u_1v_i)$. Thus $\sigma(u_1w)=\sigma(u_1w_1)=\sigma(u_1v_i)$ for $1\leqslant i\leqslant n-5$. Suppose $\sigma(uw_1)=+1$. If $\sigma(u_1w)=\sigma(u_1w_1)=\sigma(u_1v_i)=+1$, then $\dot{G}\sim\dot{G_2}(1,n-5)$. If $\sigma(u_1w)=\sigma(u_1w_1)=\sigma(u_1v_i)=-1$, then we do switching operation at $\{u_1\}$ and $\dot{G}\sim\dot{G_3}(1,n-5)$. Suppose $\sigma(uw_1)=-1$. If $\sigma(u_1w)=\sigma(u_1w_1)=\sigma(u_1v_i)=+1$, then $\dot{G}\sim\dot{G_4}(1,n-5)$. If $\sigma(u_1w)=\sigma(u_1w_1)=\sigma(u_1v_i)=-1$, then we do switching operation at $\{u_1\}$ and $\dot{G}\sim\dot{G_5}(1,n-5)$.

Now suppose $c\geqslant 2$. For any two vertices, say $u_i$ and $u_j$, in $N[u]\setminus N[v]$, then $\dot{G}[\{u,w,u_i,u_j\}]$ $\sim (K_4,+)$. Hence, $\sigma(wu_i)=+1$ for $1\leqslant i\leqslant c$, and then $\dot{G}[N[u]\setminus N[v]]$ is a clique with all positive edges. Noting that $\dot{G}[\{w,u_i,v_j,v_t\}]\sim(K_4,+)$, then $\sigma(u_iv_j)=+1$ for $1\leqslant i\leqslant c$ and $1\leqslant j\leqslant d$. By the fact that the cliques $\dot{G}[\{w_1,u_i,v_j,v_t\}]\sim (K_4,+)$ and $\dot{G}[\{u,w_1,u_i,u_j\}]\sim (K_4,+)$, we have $\sigma(w_1u_i)=+1$ for $1\leqslant i\leqslant c$ and $\sigma(uw_1)=+1$, respectively. Therefore, $\dot{G}\sim \dot{G_2}(c,d)$ in this subcase.

\textbf{Subcase 2.2.} $|N(u)\cap N(v)|\geqslant3$.

If $c=1$, then we may set $\dot{G}^*=\dot{G}_{\{u\}}$, namely, $\dot{G}^*$ is obtained from $\dot{G}$ by a switching operation at the vertex $u$. Then in $\dot{G}^*$, we have $|N(u)\cap N(w)|=2$, $\sigma(u w)=-1$, $\{w_1,w_2,\cdots,w_k\}=N[u]\setminus N[w]$, and $\{v_1,\cdots,v_d\}=N[w]\setminus N[u]$. Thus $\dot{G}^*$ satisfies the conditions of Subcase 2.1. Hence, $\dot{G}\sim\dot{G^*}\sim \dot{G_2}(k,d)$.

Now suppose $c\geqslant 2$. The fact $q=n-3$ ensures that $\dot{G}[\{w_1,\cdots,w_k\}]$ is a clique. Noting that $\sigma(vw_t)=+1$, we have $\sigma(uw_t)=-1$ for $1\leqslant t\leqslant k$. Otherwise, if $\sigma(uw_1)=+1$, then $\dot{G}[\{u,v,w_1,w_2\}]$ is an unbalanced $K_4$, which is a contradiction. Furthermore, $\dot{G}[\{w_1,\cdots,w_k\}]$ is a clique with all positive edges.
Since $\dot{G}[\{v,w,v_i,v_j\}]\sim (K_4,+)$, we have $\sigma(w v_i)=+1$ and $\sigma(v_jv_t)=+1$ for $1\leqslant i,j,t\leqslant d$. Similarly, we have $\sigma(wu_i)=+1$ and $\sigma(u_ju_t)=+1$ for $1\leqslant i,j,t\leqslant c$. Since $\dot{G}[\{w,u_i,v_j,v_t\}]\sim(K_4,+)$, we have $\sigma(u_i v_j)=+1$ for $1\leqslant i\leqslant c$ and $1\leqslant j\leqslant d$. Considering the cliques $\dot{G}[\{u,u_i,w_j,w_t\}]$ and $\dot{G}[\{v,v_i,w_j,w_t\}]$, we have $\sigma(v_i w_j)=+1$ for $1\leqslant i\leqslant d$ and $1\leqslant j\leqslant k$, and $\sigma(u_i w_j)=-1$ for $1\leqslant i\leqslant c$ and $1\leqslant j\leqslant k$. While $\dot{G}[\{w_1,w_2,u_1,v_1\}]$ is an unbalanced $K_4$, which is a contradiction.
\end{proof}
\section{A proof of spectral Tur$\mathbf{\acute{a}}$n number of $\mathcal{K}_4^-$}
Let $\dot{G}$ be a signed graph of order $n$. By the table of the spectra of signed graphs with at most six vertices \cite{BCST}, we can check that Theorem \ref{spectrum} is true for $n\leqslant 6$. Therefore, assume that $n\geqslant 7$. The following celebrated upper bound of $\rho(G)$ is very crucial for our proof.

\begin{theorem}\cite{HSF, N}\label{delta}
Let $G$ be a graph of order $n$ with the minimum degree $\delta=\delta(G)$ and $e=e(G)$. Then
$$\rho(G)\leqslant\frac{\delta-1+\sqrt{8e-4\delta n+(\delta+1)^2}}{2}.$$
\end{theorem}

The negation of $\dot{G}$ (denoted by $-\dot{G}$) is obtained by reversing the sign of every edge in $\dot{G}$. Obviously, the eigenvalues of $-\dot{G}$ are obtained by reversing the sign of the eigenvalues of $\dot{G}$.

\begin{proof}[\rm{\textbf{Proof of Theorem \ref{spectrum}.}}]
Let $\dot{G}=(G,\sigma)$ be a $\mathcal{K}_4^-$-free unbalanced signed graph with maximum spectral radius. Since $\dot{G_1}(0,n-3)$ is a $\mathcal{K}_4^-$-free unbalanced signed graph, by Lemma \ref{largest},
\begin{equation}\label{n-2}
\rho(\dot{G})\geqslant \rho(\dot{G_1}(0,n-3))=n-2.
\end{equation}

We claim that $\rho(\dot{G})=\lambda_1(\dot{G})$. Otherwise $\rho(\dot{G})=\max\{\lambda_1(\dot{G}),-\lambda_n(\dot{G})\}=-\lambda_n(\dot{G})$.  Assume that $\dot{G_1}=-\dot{G}$. Hence, $\lambda_1(\dot{G_1})=-\lambda_n(\dot{G})$. Since $\dot{G}$ is $\mathcal{K}_4^-$-free, we have $\omega_b(\dot{G_1})\leqslant 3$. By Lemma \ref{balanced}, for $n\geqslant 7$, we have
$$\rho(\dot{G})=-\lambda_n(\dot{G})=\lambda_1(\dot{G_1})\leqslant\left(1-\frac{1}{\omega_b(\dot{G_1})}\right)n\leqslant\frac{2}{3}n<n-2,$$
which is a contradiction to (\ref{n-2}).

We claim that $\dot{G}$ is connected. Let $\bm{x}=(x_1,x_2,\cdots,x_n)^T$ be a unit eigenvector of $A(\dot{G})$ corresponding to $\lambda_1(\dot{G})$. Let vertices $u$ and $v$ be any two vertices belonging to distinct components, then we construct a signed graph $\dot{G_2}=(G+uv,\sigma_2)$ with $\sigma_2(e)=\sigma(e)$ when $e\in E(\dot{G})$. If $x_ux_v\geqslant 0$ (resp. $x_ux_v< 0$), then we take $\sigma_2(uv)=+1$ (resp. $\sigma_2(uv)=-1$). Hence, by Rayleigh-Ritz Theorem we have
\begin{equation}\label{xu}
\lambda_1(\dot{G}_2)-\lambda_1(\dot{G})\geqslant\bm{x}^TA(\dot{G}_2)\bm{x}-\bm{x}^TA(\dot{G})\bm{x}= 2\sigma_2(uv)x_ux_v\geqslant0.
\end{equation}
Since $\dot{G}_2$ is also $\mathcal{K}_4^-$-free and unbalanced, we have $\lambda_1(\dot{G}_2)\leqslant\lambda_1(\dot{G})$. So, $\lambda_1(\dot{G}_2)=\lambda_1(\dot{G})$. Furthermore, $\lambda_1(\dot{G}_2)=\bm{x}^TA(\dot{G}_2)\bm{x}$ holds, and then $A(\dot{G}_2)\bm{x}=\lambda_1(\dot{G}_2)\bm{x}$. From (\ref{xu}), $x_ux_v=0$ holds. Without loss of generality, suppose $x_u=0$. By $A(\dot{G})\bm{x}=\lambda_1(\dot{G})\bm{x}$ and $A(\dot{G}_2)\bm{x}=\lambda_1(\dot{G}_2)\bm{x}$, we have
$$\lambda_1(\dot{G})x_u=\sum_{w\in N_{\dot{G}}(u)}\sigma(uw)x_w=0,$$
and then
$$\lambda_1(\dot{G}_2)x_u=\sum_{w\in N_{\dot{G_2}}(u)}\sigma(uw)x_w+\sigma_2(uv)x_v=\sigma_2(uv)x_v=0.$$
Hence, $x_v=0$. Then $\bm{x}$ is a zero vector, which is a contradiction.

We claim that $\delta(\dot{G})\geqslant 2$. Otherwise there exits a vertex $u$ with $d_{\dot{G}}(u)=1$ and $uv\notin E(\dot{G})$ for some vertex $v$. Then we construct a signed graph $\dot{G_3}=(G+uv,\sigma_3)$. Noting that $d_{\dot{G_3}}(u)=2$, $\dot{G_3}$ is still $\mathcal{K}_4^-$-free unbalanced. Furthermore, we may obtain a contradiction as above.

If $e(\dot{G})\leqslant\frac{n(n-1)}{2}-(n-2)$, for the underlying graph $G$, by Theorem \ref{delta} we have
\begin{align*}
\rho(G)&\leqslant\frac{\delta-1+\sqrt{8e(G)-4\delta n+(\delta+1)^2}}{2}\\
&\leqslant\frac{\delta-1+\sqrt{8(\frac{n(n-1)}{2}-(n-2))-4\delta n+(\delta+1)^2}}{2}\\
&=\frac{\delta-1+\sqrt{4n^2-4(\delta+3)n+\delta^2+2\delta+17}}{2}\\
&\leqslant\frac{\delta-1+\sqrt{4n^2-4(\delta+3)n+\delta^2+6\delta+9}}{2}\\
&=n-2.
\end{align*}
By Lemma \ref{unbalanced}, we have $\rho(\dot{G})=\lambda_1(\dot{G})<\rho(G)\leqslant n-2$, which is a contradiction to (\ref{n-2}). Hence, $e(\dot{G})\geqslant\frac{n(n-1)}{2}-(n-3)$. By Theorem \ref{edge}, we have $e(\dot{G})=\frac{n(n-1)}{2}-(n-3)$ and $\dot{G}$ is switching equivalent to a signed graphs in $\mathcal{G}$. Noting that $\rho(\dot{G})\geqslant n-2$, by Lemma \ref{largest}, we have $\dot{G}\sim \dot{G_1}(0,n-3)$, and $\rho(\dot{G})= n-2$.
\end{proof}

\textbf{Conflicts of Interest:} The authors declare no conflict of interest.


\begin{thebibliography}{100}

\bibitem{ABH} S. Akbari, F. Belardo, F. Heydari, et al., On the largest eigenvalue of signed unicyclic graphs, Linear Algebra Appl. 581 (2019) 145-162.

\bibitem{BCKW}  F. Belardo, S. Cioab\u{a}, J. Koolen, J. Wang, Open problems in the spectral theory of signed graphs, Art Discrete Appl. Math. 1 (2018) P2. 10.

\bibitem{BCST} F. C. Bussemaker, P. J. Cameron, J. J. Seidel, et al., Tables of signed graphs, Eut Report 91-WSK-01, Eindhoven, 1991.

\bibitem{BH} A. E. Brouwer, W. H. Haemers, Spectra of graphs, Springer, 2011.

\bibitem{BS} J. A. Bondy, M. Simonovits, Cycles of even length in graphs, J. Combin. Theory Ser. B 16 (1974) 97-105.

\bibitem{CDT} S. Cioab\u{a}, D. N. Desai, M. Tait, The spectral radius of graphs with no odd wheels, European J. Combin. 99 (2022) 103420.

\bibitem{CH} D. Cartwright, F. Harary, Structural balance: a generalized of Heider’s theory, Psyshol Rev. J. 63 (1956) 227-293.

\bibitem{DKL} D. N. Desai, L. Y. Kang, Y. T. Li, et al., Spectral extremal graphs for intersecting cliques, Linear Algebra Appl. 644 (2022) 234-258.

\bibitem{FG} Z. F\"{u}redi, D. S. Gunderoson, Extremal numbers for odd cycles, Comb. Probab. Comput. 24 (2015) 641-645.

\bibitem{H} F. Harary, On the notion of balance of a signed graph, Michigan Math. J2 (1953) 143-146.

\bibitem{H1} H. Huang, Induced graphs of the hypercube and a proof of the Sensitivity Conjecture, Ann. of Math. 190 (2019) 949-955.

\bibitem{HSF} Y. Hong, J. Shu, K. Fang, A sharp upper bound of the spectral radius of graphs, J. Combin. Theory Ser. B 81 (2001) 177-183.

\bibitem{KP} M. R. Kannan, S. Pragada, Signed spectral Tur$\rm{\acute{a}}$n type theorems, https://arxiv.org /abs/2204.09870.

\bibitem{KS} T. Koledin, Z. Stani$\rm{\acute{c}}$, Connected signed graphs of fixed order, size, and number of negative edges with maximal index, Linear and Multilinear Algebra 65 (2017) 2187-2198.

\bibitem{N} V. Nikiforov, Some inequalities for the largest eigenvalue of a graph, Combin. Probab. Comput. 11 (2002) 179-189.

\bibitem{N1} V. Nikiforov, Bounds on graph eigenvalues II, Linear Algebra Appl. 427 (2007) 183-189.

\bibitem{S} Z. Stani$\rm{\acute{c}}$, Bounding the largest eigenvalue of signed graphs, Linear Algebra Appl. 573 (2019) 80-89.

\bibitem{T} P. Tur$\rm{\acute{a}}$n, On an extremal problem in graph theory, Matematikaiés Fizikai Lapok (in Hungarian) 48 (1941) 436-452.

\bibitem{W} H. S. Wilf, Spectral bounds for the clique and independence numbers of graphs, J. Combin. Theory Ser. B 40 (1986) 113-117.

\bibitem{WHL} D. J. Wang, Y. P. Hou, D. Q. Li, Extremed signed graphs for triangle, https://arxiv. org/abs/2212.11460.

\bibitem{WKX} J. Wang, L. Y. Kang, Y. S. Xue, On a conjecture of spectral extremal problems, J. Combin. Theory Ser. B 159 (2023) 20-41.

\bibitem{WYQ} W. Wang, Z. D. Yan, J. G. Qian, Eigenvalues and chromatic number of a signed graph, Linear Algebra Appl. 619 (2021) 137-145.

\bibitem{YZ} L. T. Yuan, X. D. Zhang, Tur$\rm{\acute{a}}$n numbers for disjoint paths, J. Graph Theory 98 (2021) 499-524.

\bibitem{Z} T. Zaslavsky, Signed graphs, Discrete Appl. Math. 4 (1982) 47-74.



\end{thebibliography}
\end{document}